\documentclass[11 pt]{amsart}
\usepackage{amsmath,amssymb,amsthm, mathrsfs, mathtools}
\usepackage{amscd,mathtools}
  \usepackage{geometry}
  \usepackage{graphicx,epstopdf}
  \usepackage{color,xcolor}
  \usepackage{enumerate}
  \usepackage{xspace}
  \usepackage{tipa}
  \usepackage{epsfig}
  \usepackage{pinlabel}
  \usepackage[all]{xy}

\usepackage{tikz}
\usepackage{caption, subcaption}
\usepackage{tikz-cd}
\usetikzlibrary{calc,decorations.pathreplacing}
\usepackage{mathrsfs}
\usetikzlibrary{matrix}
\usetikzlibrary{positioning}
\DeclareMathAlphabet{\pazocal}{OMS}{zplm}{m}{n}
\tikzset{>=stealth}
  \usepackage{hyperref}

  \newcommand{\calD}{\mathcal{D}}

  \newcommand{\calT}{\mathcal{T}}

  \newcommand{\FF}{\mathbb{F}}

  \newcommand{\NN}{\mathbb{N}}

  \newcommand{\RR}{\mathbb{R}}

  \newcommand{\ZZ}{\mathbb{Z}}

  \newtheorem{theorem}{Theorem}[section]
  \newtheorem{proposition}[theorem]{Proposition}
  \newtheorem{corollary}[theorem]{Corollary}
  \newtheorem{lemma}[theorem]{Lemma}

  \newtheorem{introthm}{Theorem}

  \theoremstyle{definition}
  \newtheorem{definition}[theorem]{Definition}
  \newtheorem{claim}[theorem]{Claim}
  
  \newtheorem*{claim*}{Claim}

  \newtheorem*{question*}{Question}
  \newtheorem*{answer*}{Answer}
  \newtheorem*{application*}{Application}

  \theoremstyle{remark}
  \newtheorem{remark}[theorem]{Remark}
  \newtheorem*{remark*}{Remark}
  
  \newcommand{\Aut}{\ensuremath{\operatorname{Aut}}\xspace} 
   \newcommand{\GL}{\ensuremath{\operatorname{GL}}\xspace} 
   \newcommand{\oG}{{\overline \Gamma}}

  \newcommand{\supp}{{\rm supp}}

\DeclareMathOperator{\out}{Out}

\newcommand{\Out}{\out(\F)}
  \newcommand{\param}{{\mathchoice{\mkern1mu\mbox{\raise2.2pt\hbox{$
  \centerdot$}}
  \mkern1mu}{\mkern1mu\mbox{\raise2.2pt\hbox{$\centerdot$}}\mkern1mu}{
  \mkern1.5mu\centerdot\mkern1.5mu}{\mkern1.5mu\centerdot\mkern1.5mu}}}
  \renewcommand{\setminus}{{\smallsetminus}}

  \newcommand{\F}{{\FF_n}} 
  
    \newcommand{\A}{{A(\Gamma)}}

\begin{document}

%section{Title and abstract}

\title[Dilatation of outer automorphisms of Right-angled Artin Groups]{Dilatation of outer automorphisms of Right-angled Artin Groups }
  
%%authors
%
\author   {Corey Bregman}
\address{Department of Mathematics, Brandeis University}
\email{Cbregman@brandeis.edu}
\author   {Yulan Qing}
\address{Department of Mathematics, University of Toronto, Toronto, ON }
\email{yulan.qing@gmail.com}

%\ thanks{The first Authors thanks whoever helped.}
 
  \date{\today}
%  \subjclass[2010]{57M50 (32G15, 37D40, 37A25)}

\begin{abstract} 
%We study the dilatation of outer automorphisms of right-angled Artin groups. 

Given a right-angled Artin group defined by a simplicial graph: $A(\Gamma) = \langle V | E \rangle$ and an automorphism $\phi \in \Aut(A(\Gamma))$ there is a natural measure of how fast the length of a word $w$ of $A(\Gamma)$ grows after $n$ iterations of $\phi$ as a function of $n$, which we call the dilatation of $w$ under $\phi$. We define the dilatation of $\phi$ as the supremum over dilatations of all $w \in \A$. Assuming that $\phi$ is a pure and square map, we show that if the dilatation of $\phi$ is positive, then either there exists a free abelian special subgroup on which the dilatation is realized; or there exists a strata of either free or free abelian groups on which the dilatation is realized. 

\end{abstract}
  
 \maketitle
  
  %\tableofcontents

\section{Introduction}
Let $\Gamma=(V,E)$ be a finite simplicial graph.  The right-angled Artin group (RAAG) associated to $\Gamma$ is the group with presentation
\[A(\Gamma)=\langle v\in V|[v, w], \mbox{ }(v,w)\in E\rangle\]
That is, the generators are in bijection with $V$, and the only relations are that two generators commute if the corresponding vertices form an edge in $\Gamma$. From their simple presentations, we see that RAAGs encompass a spectrum of groups with free groups $F_n$ at one end and free abelian groups $\ZZ^n$ at the other. RAAGs are a well-studied class of groups with close connections to low-dimensional topology; they arise naturally in the study of diffeomorphism groups of manifolds \cite{KK17} and most recently, they played a key role in Agol's solution \cite{Ag13} to the virtually Haken conjecture for hyperbolic 3-manifolds.

In this article, we study the automorphism groups of RAAGs.  At one end of the spectrum, the study of $\GL_n(\ZZ)$ is classical, because of its relation to classification of lattices in $\RR^n$ and in turn, flat metrics on tori.  At the other, $\Out$, or the outer automorphism group of the free group, became an active area of research beginning in the early 20th century with the work of Nielsen and Magnus \cite{LS77}.  The introduction of Culler--Vogtmann Outer space \cite{CV86} provided a geometric action for $\Out$, putting its study on equal footing with that of lattices in semisimple Lie groups and perhaps more closely with the mapping class group of an orientable surface of genus $g$.  

Let $\out(A(\Gamma))$ denote the outer automorphism group of $A(\Gamma)$.  $\out(A(\Gamma))$ has been shown to possess many of the same properties as $\Out$ and $\GL_n(\ZZ)$. For example, it is known to be finitely generated \cite{Ser89},\cite{Lau95}, finitely presented \cite{Day09}, and to have finite virtual cohomological dimension \cite{ChVo09}, \cite{BCV09}, \cite{DW18}.  

In what follows, we consider the algebraic and dynamical structure of automorphisms of RAAGs, focusing specifically on the asymptotic growth of conjugacy classes of elements under the action of an automorphism.  Growth of automorphisms of $\GL_n(\ZZ)$ can be easily obtained from the Jordan normal form of a matrix.  For $\Out$, on the other hand, a characterization of the growth is obtained by demonstrating the existence of a train track map or a relative train track map, first constructed in \cite{BH92}, \cite{BFH00}, \cite{BFH05}.  In either the free or free abelian case, words grow either polynomially or exponentially under the iterated application of an automorphism. 

By work of Day \cite{Day09}, any automorphism of $A(\Gamma)$ decomposes uniquely into an $\FF_n$-part and a $\ZZ^n$-part, and one may hope that the dynamics of any automorphism admits a similar structural decomposition. However, the construction of train tracks in the free group case relies heavily on Culler--Vogtmann outer space.  For certain classes of RAAGs, an outer space has been constructed \cite{CCV07}, and Charney--Stambaugh--Vogtmann \cite{CSV17} have constructed an outer space for the ``$\FF_n$"-part of $\out(A(\Gamma))$, but no outer space yet exists in general. To overcome this difficulty, we restrict our attention to the collection
$\Phi\in \out(A(\Gamma))$ which map 2-cells to 2-cells without folding the boundary edges (see Definition~\ref{Def:Square}), and call such $\phi$ square maps.  

Recall that an element $\Phi\in \out(\A)$ of is an equivalence class of automorphisms $\Phi=[\phi]$, defined up to conjugation by elements of $\A$.  Given an automorphism $\phi\in \Aut(A(\Gamma))$ and a word $w\in A(\Gamma)$, we define the dilatation $\lambda_\phi(w)$ to be the logarithm of the length of $\phi^n(w)$ divided by $n$.  The dilatation captures the average exponential growth of the length of $w$ under iterated application of $\phi$. The \textit{dilatation} $\lambda_\phi$ is the supremum of $\lambda_\phi(w)$ over all words $w\in A(\Gamma)$.  

%If there exists a finite sequence of distinct generators $\{ s_{1}, s_{2}, s_{3}...s_{k} \}, k \geq 3$ such that $s_{i} \in supp(\phi(s_{i-1}))$, then we say each $s_{i}$ is a \emph{circular} element (of the generating set). As a partial confirmation of the to the question above, we show 
\begin{introthm}\label{PositiveDilatation}
Let  $\Phi=[\phi] \in \out(\A)$ be a pure square map. If the dilatation $\lambda_{\phi}$ is positive, then there exists a $\phi$--invariant induced subgraph $\Delta \subset \Gamma$ where 
$\lambda_{\phi|_{\Delta}} = \lambda_{\phi}$ and $\Delta$ satisfies one of the following:
\begin{enumerate}[(i)]
\item $\Delta$ is a complete graph.
\item $\Delta$ contains an empty graph $\Delta_{e}$ that is the union of all cycles in $\calD_{\Delta}$.
\end{enumerate}
\end{introthm}
To prove this theorem we define a finite directed graph $\calD_{\Delta}$ which we call the automorphism diagram (cf. \S3). The automorphism diagram which keeps track of the verbal substitution of each generator of $A(\Gamma)$ under iterated application of $\phi$.
 
Theorem \ref{PositiveDilatation} can be interpreted as saying that the growth of square automorphisms is not more complicated than that of its free and free abelian subgroups. If the dilatation of a word in $A(\Gamma)$ is zero, this means that its length grows subexponentially. It follows from relative train track theory for free groups \cite{BH92},\cite{BFH00} and from the Jordan normal form for free abelian groups that if an automorphism has vanishing dilatation, then it grows polynomially.  We conjecture that the same is true for general RAAG automorphisms, and in Theorem \ref{Lem:cycle}, we prove that if a pure, square map $\phi$ has an automorphism diagram without cycles, then $\lambda_\phi=0$ and $\phi$ grows polynomially.

%As a refinement of Theorem \ref{PositiveDilatation}, we show that each generator of $\A$ either grows exponentially or polynomially. It is somewhat surprising that a generalization of $GL(n, \ZZ)$ and $\Out$ follow the same dichotomy of the growth type of its elements:
%\begin{introthm}\label{Decomposition}
%Let $\phi \in Out(A(\Gamma))$ be a pure, square map. Each generator $s_{i} \in \A$ either grows exponentially or polynomially under iterations of $\phi$. 
%  \end{introthm}

The paper is organized as follows.  In \S2, we review relevant background information on centralizers in right-angled Artin groups and automorphisms, and define the notions of dilatation of an automorphism and of a square map for right-angled Artin groups. In \S3, we introduce the main technical tool of the paper, the automorphism diagram, and analyze its key properties for square maps. We show that an automorphism with acyclic automorphism diagram has polynomial growth in  \S4. Finally in \S5,   we apply the results of \S4 to find the system of free and free abelian subgraphs in Theorem \ref{PositiveDilatation}.
% and \ref{Decomposition}. 

\section{Background}
\subsection{Finite simple graphs}
Let $\Gamma=(V,E)$ be a \emph{finite simple graph}. That is, a graph in which each edge is uniquely determined by a pair of distinct vertices.  By an \emph{induced subgraph} we mean a subgraph $\Gamma'= (V',E')$ such that $V' \subseteq V$ and given $v_1, v_2 \in V'$
\[
(v_1, v_2) \in E' \Longleftrightarrow (v_1, v_2) \in E.
\]
Given a vertex $v_0 \in V$, the \emph{link} of $v_0$ is the induced subgraph on the set of vertices 
\[
\{v | (v, v_0) \in E\}.
\]
ln contrast, the \emph{star} of a vertex $v_0 \in V$ is the induced subgraph on 
\[
\{v | (v, v_0) \in E\} \cup \{v_0\}.
\]
The \emph{complement} of $\Gamma = (V, E)$, denoted $\overline{\Gamma}$, is the graph $(V, E')$ where
\[
(v_1, v_2) \in E' \Longleftrightarrow (v_1, v_2) \notin E.
\]
\subsubsection{Directed graphs}
By a \emph{directed graph} we mean a finite graph where each edge is an ordered pair $(v_i, v_j)$ and $(v_i, v_j) \neq (v_j, v_i)$. We also write $v_i \to v_j$ for a directed edge from $v_i$ to $v_j$. We can depict a directed graph by putting an arrow on each edge.
A path $\gamma$ is a sequence of vertices  
\[
v_0 \to v_1 \to v_2... \to v_n
\]
A \emph{cycle} is a path as above such that $v_0=v_n$. We  will refer to  a path by the tuple of vertices \emph{i.e.} $\gamma=(v_0,\ldots,v_n)$. 
\begin{remark}
In contrast to standard definition of cycle in graph theory, we allow repeated vertices. 
\end{remark}

\subsection{Right-angled Artin groups}
Let $A(\Gamma) = \langle V| E \rangle$ denote the \emph{right-angled Artin group} associated to $\Gamma$.  The generators of $A(\Gamma)$ are in one-to-one correspondence with the vertices $V$, and two generators commute if and only if their corresponding vertices are connected by an edge in $\Gamma$.  A \emph{special subgroup} of $A(\Gamma)$ is any subgroup of the form $A(\Delta)$, where $\Delta \subseteq \Gamma$ is an induced subgraph.  

\begin{definition}
We fix the generating set once and for all. Given word $w \in \A$,  we say $w$ is \emph{reduced} if it cannot be shortened by successively applying commutation relations and canceling $vv^{-1}$ pairs, $v\in V$.  We say $w$ is \emph{cyclically reduced} if $w$ and all of its cyclic conjugates are reduced. If $w$ is a reduced, then the \emph{length} of $w$, denoted $|w|$, is the number of letters in the spelling of $w$, and the \emph{support} of $w$, denoted $supp(w)$, is the set of elements $v$ in  $V$ such that $v$ or $v^{-1}$ occurs in the spelling of $w$.   
%We use $[w]$ to denote the set of representatives of the word $w$ in $\A$.

\end{definition}

\subsection{Commuting elements in $\A$}
$\A$ is a special case of a graph product, where all the vertex groups are infinite cyclic. The general construction proceeds as follows. Let $\Gamma=(V,E)$ be a graph. To each $v \in V$, associate a vertex group $G_v$. Then, the \emph{graph product} of the $\{G_v\}$ with respect to the graph $(V, E)$ is defined as $F/R$ where $F$ is the free product of the  $G_v$ and $R$ is the normal subgroup generated by subgroups of the form $[G_u,G_v]$ whenever there is an edge joining $u$ and $v$.

In \cite{Ser89},  Servatius studied graph products of groups and characterized when the support of two words commute pairwise. Given two words $w_1,w_2$, we will say that $supp(w_1)$ \textit{commutes with} $supp(w_2)$ if the elements of $supp(w_1)$ and $supp(w_2)$ pairwise commute.We recall the relevant definitions and results here:

\begin{definition}\label{factorial}
If $u ,v \in A(\Gamma)$ are reduced, then 
\[ |uv| \leq |u| + |v|. 
\]
and if equality holds we say that the product $uv$ is a \emph{reduced factorization}. If $uv$ is not reduced, then there exists a unique $h$ such that we have that 
\[
u = u'h  \text{ with } |u| = |u'| + |h| \\
v = h^{-1}v'  \text{ with } |v| = |v'| + |h|
 \]
 and 
 
 \[
 u v = u'v'  \text{ with } |uv| = |u'| + |v'|.
 \]
 \end{definition}

\begin{proposition}\label{commutingwords}
Suppose that $u$ and $v$ are commuting elements in $A(\Gamma)$ such that
\begin{enumerate}
\item $uv$ is a reduced factorization, and
\item the induced subgraph of $\Gamma $ on $supp(uv)$ has connected complement.
\end{enumerate}

Then there exists an element $h$ in $A(\Gamma)$ such that both $u$ and $v$ belong to the cyclic subgroup of $A(\Gamma)$ generated by $h$.
\end{proposition}

\begin{lemma}\label{complement}
If a graph $\Gamma = (V, E)$ has its complement $\oG$ disconnected, then the graph is a join. That is to say, there exists a partition of vertices into two set $V_1$, $V_2$, such that for all $v_1\in V_1$ and $v_2 \in V_2$,
\[
(v_1, v_2) \in E
\] 
\end{lemma}
\begin{proof}
Since $\oG$ is disconnected, the edge set of $\oG$, we call $\overline E$ can be partitioned into two sets $E_1$ and $E_2$ such that all edges in $E_i$ connect vertices in $V_i$. That is to say, all the edges connecting vertices from $V_1$ to $V_2$ 
is in the complement of $\oG$, i.e. for every $v_1 \in V_1$ and $v_2 \in V_2$, 
\[
(v_1, v_2) \in \overline \oG = \Gamma.
\] 
That is to say there is a partition of vertices into two set $V_1$, $V_2$, such that for all $v_1\in V_1$ and $v_2 \in V_2$,
\[
(v_1, v_2) \in E.
\] 
\end{proof}

\subsection{Out(RAAG)}

Let $\Aut(A(\Gamma))$ be the automorphism group of $A(\Gamma)$. $\Aut(A(\Gamma))$ is finitely generated \cite{Ser89},\cite{Lau95} by the Laurence--Servatius generators, which fall into the following four types:
\begin{itemize}
\item Inversions: $\iota_v:v\mapsto v^{-1}$ and is the identity on all other generators.
\item Graph isomorphisms: any automorphism of $\Gamma$ induces a bijection of $V$ which preserves commutation relations,
\item Transvections: if $lk(v)\subseteq st(w)$, \[\tau_{v,w}:v\mapsto wv\] and fixes all other generators.   $\tau_{v,w}$ is called a \textit{twist} if $v$ and $w$ and a \textit{fold} otherwise.  
\item Partial conjugations: if $C$ is a component of $\Gamma\setminus st(w)$, \[\chi_{C,w}:v\mapsto wvw^{-1}\] and fixes all other generators.  
\end{itemize}
We let $\text{Out}(A(\Gamma))$ denote the outer automorphism group of $A(\Gamma)$; this is the quotient of $\Aut(A(\Gamma))$ by the normal subgroup of automorphisms generated by the action of $A(\Gamma)$ on itself by conjugation. $\text{Out}(A(\Gamma))$ is generated by images of the above Laurence--Servatius generators in the quotient.  
\begin{remark}
In what follows we will often define properties of elements $\Phi\in\out(\A)$. Since $\Phi=[\phi]$ is strictly speaking an equivalence class, when we say $\Phi$ has some property $\mathcal{P}$, we will mean that there is some representative $\phi \in\Aut(\A)$ which has property $\mathcal{P}$.  As an abuse of notation, we will not distinguish between $\phi$ and $[\phi]$, unless it is not clear from the context.
\end{remark}
\subsubsection{Pure automorphisms of $A(\Gamma)$}
In this section we introduce the notion of a \emph{pure} automorphism in $\out(\A)$. An element $\phi\in\Aut(\A)$  is \emph{pure} if it is cyclically reduced and 
\[ s \in supp(\phi(s)) \text{ for all } s\in V.\]
An equivalence class $[\phi]\in \out(\A)$ is pure if it has a pure representative in $\Aut(\A)$.
\begin{remark}\label{purepower}
If $\phi$ is a pure automorphism, then it follows from definition same representative $\phi^{k}$ is also a pure automorphism.
\end{remark}
\begin{proposition} \label{pure}
For every $\phi\in \Aut(\A)$, there exists a positive integer $N$ such that $\phi^N$ is a pure automorphism.
\end{proposition}
\begin{proof}
The abelianization of $A(\Gamma)$ is obtained from $A(\Gamma)$ by making all generators commute.  Hence $A(\Gamma)^{ab}\cong \ZZ^{|V|}$, generated by the equivalence classes $\{[v]|v\in V\}$. For any word $w\in A(\Gamma)$ we write \[[w]=\sum_{v\in V}n_v[v]\in \ZZ^{|V|}\]where $n_v$ is just the exponent sum of all $v$'s occurring in $w$. By reducing coefficients modulo 2, there is a natural surjection $\ZZ^{|V|}\rightarrow (\ZZ/2\ZZ)^{|V|}$. Observe that if $n_v\not\equiv 0 \mbox{ (mod 2)}$, then we must have $v\in supp(w)$.  

Since the kernel of the surjection $A(\Gamma)\rightarrow (\ZZ/2\ZZ)^{|V|}$ is characteristic, any automorphism $\phi$ of $A(\Gamma)$ induces an automorphism $A_\phi\in \GL(|V|,\ZZ/2\ZZ)$. The latter has finite order, hence for some $N$, we have $A_\phi^N=\text{Id}$. But this means that for any $s\in V$,  if we write \[[\phi^N(s)]=\sum_{v\in V}n_v[v], \] then \[n_v\equiv \left\{\begin{array}{c}  0 \mbox{ (mod 2)}, \mbox{ $v\neq s$}\\
1 \mbox{ (mod 2)}, \mbox{ $v= s$}
\end{array}\right.\]
Hence $\phi^N$ is pure, as desired.
\end{proof}

 \subsection{Dilatation}
 
 Consider an infinite order automorphism $\phi$ of $A(\Gamma)$. We will be interested in the dynamics of $\phi$ acting on $A(\Gamma)$.  One way to do this is to measure the asymptotic growth of the word length of elements of $A(\Gamma)$ under powers of $\phi$.
 \begin{definition}
 The \textit{dilatation of $\phi$ at $w$} is 
\[
\lambda_{\phi} (w): =\lim_{k \to \infty} \frac{\log|\phi^k(w)|}{k},
\]
and the \textit{dilatation of $\phi$} is 
\[
 \lambda_{\phi}  : =\sup _{w \in A(\Gamma)} \lambda_{\phi} (w)
\]
\end{definition}
Note that $\lambda_{\phi} (w)$ always exists since the sequence $\left\{\frac{\log|\phi^k(w)|}{k}\right\}$ is subadditive. We claim that the dilatation $\lambda_\phi$ is always realized by some word $w$, and in fact, we can take $w$ to be a generator:

\begin{lemma}\label{Lem:dilatationofgenerator}

\[
\lambda_{\phi}=\max_{s\in V}\{\lambda_\phi(s)\}
\]In particular, $\lambda_{\phi}$ is always finite and realized by some (not necessarily unique) generator $s_{i}$. 
\end{lemma}

\begin{proof}
Given any word $w$ in $A(\Gamma)$ we can write $w = s_1 s_2\cdots s_n$ for some generators $s_i\in V$.  Then
\begin{align*}|\phi^k(w)|&=|\phi^k(s_1)\cdots\phi^k(s_n)|\\
&\leq |\phi^k(s_1)|+\cdots+|\phi^k(s_n)|\\
&\leq n\max_i\{|\phi^k(s_i)|\}
\end{align*}
Taking log of both sides and dividing by $k$:
\[
\frac{\log|\phi^k(w)|}{k}\leq \frac{\log(n)}{k}+\max_i\left\{\frac{\log|\phi^k(s_i)|}{k}\right\}
\]
Since $n$ is fixed, the first term on the right goes to 0 as $k\rightarrow \infty$. We conclude that for any $w\in A(\Gamma)$ \[\lambda_\phi(w)\leq \max_{s\in V}\{\lambda_\phi(s)\}\] and therefore 
\[\lambda_\phi=\max_{s\in V}\{\lambda_\phi(s)\}.\]In particular, $\lambda_\phi$ exists and is realized by a generator.  

\end{proof}
Given $g\in A(\Gamma)$, let $\phi^g$ denote the composition of $\phi$ with conjugation by $g$. It is not difficult to see that $\lambda_\phi=\lambda_{\phi^{g}}$. Indeed, given $w\in \A$ we have:\begin{align*}|(g\phi g^{-1})^k(w)|&=|g\phi(g)\cdots\phi^{k-1}(g)\phi^k(w)\phi^{k-1}(g^{-1})\cdots\phi(g^{-1})g^{-1}| \\
&\leq 2\left(|g|+|\phi(g)|+\cdots +|\phi^{k-1}(g)|\right)+|\phi^k(w)|\\
&\leq 2k\max\{|g|,|\phi(g)|,\ldots, |\phi^{k-1}(g)|, |\phi^k(w)|\}
\end{align*}
Now applying $\log$ to both sides and dividing by $k$, this is at most $\lambda_\phi$ in the limit, since $\log(2k)/k\rightarrow 0$.  Thus, for all $w$, $\lambda_{\phi^{g}}(w)\leq\lambda_\phi$. By symmetry, $\lambda_{\phi^g}\geq\lambda_\phi$ and the thus $\lambda_\phi=\lambda_{\phi^{g}}$. Therefore, the dilatation is a well-defined invariant of the equivalence class $[\phi]\in \text{Out}(A(\Gamma))$. If $\lambda_{\phi}(s)>0$ or $\lambda_{\phi}>0$ is positive, we also say that the element $s$ \emph{grows exponentially} under the map $\phi$, or that $\phi$ is an \emph{exponentially growing map}.

Furthermore, we say a subgraph $\Delta\subseteq \Gamma$, is \emph{invariant} under $\phi$ if $\phi$ takes words whose support is in $\Delta$ to words whose support is in $\Delta$. Given an invariant subgraph  $\Delta \subset \Gamma$, one can define 
\[
\lambda|_{\Delta} : = \sup _{w \in A(\Delta)} \lambda_{\phi} (w).
\]

\subsection{Square maps}
If we identify $\A$ with the fundamental group of a Salvetti complex $S(\Gamma)$, then elements of $\out(A(\Gamma))$ as homotopy classes of maps from $S(\Gamma)$ to itself. Pairs of commuting generators correspond to 2-tori in $S(\Gamma)$, and we will be interested in when $\Phi\in \out(A(\Gamma))$ folds these tori nicely onto other tori in $S(\Gamma)$.  This idea is encoded in the following definition.
%In this paper we assume that,

\begin{definition}\label{Def:Square}
Fixing the presentation of $A(\Gamma)$ to be the one associated with $\Gamma$ once and for all.  The map $\phi \in  \Aut(A(\Gamma))$ is called a \textit{square map} if for all $s_1,s_2 \in S$:
 %the image of commuting generators have supports that \emph{pairwise} commute, i.e.:
\[ s_1 \text{ commutes with } s_2 \qquad \Longrightarrow \qquad   supp(\phi(s_1)) \text{ commutes with } supp(\phi(s_2)). \]

We say that $\Phi\in \out(\A)$ is a square map if there exists a representative $\phi$ of $\Phi$ which is a square map.
\end{definition}

% commuting generators are mapped to words that do not
%cancel when concatenated:
%\begin{definition}\label{square map}
%An element $\phi \in \out(A(\Gamma))$ is a \emph{square map} if given any pair of commuting generators $s_{1}, s_{j} \in A(\Gamma)$ then 
%\begin{align*}
%|\phi(s_{i}) \phi(s_{j}) | = |\phi(s_{i})|+|\phi(s_{j}) |\\
%|\phi(s_{j}) \phi(s_{i}) | = |\phi(s_{i})|+| \phi(s_{j})|\\
%\end{align*}
%That is to say, both  $\phi(s_{i}) \phi(s_{j})$ and $\phi(s_{j})\phi(s_{i})$ are factorially reduced.
%\end{definition}

We say that an automorphism $\phi$ is a \emph{positive} automorphism if all generators are mapped to words that do not contain negative powers of any generator.
\begin{lemma}
A positive automorphism $\phi$ is a square automorphism.
\end{lemma}

\begin{proof}
Since $\phi(s_1)$ and $\phi(s_2)$ are primitive elements, they are not powers of any other word. By assumption $\phi(s_1)$ and $\phi(s_2)$ commute. Also, since $\phi$ is a positive map,  both $\phi(s_1)\phi(s_2)$ and $\phi(s_2)\phi(s_1)$ are factorially reduced as defined in Definition~\ref{factorial}. Therefore, by Proposition~\ref{commutingwords}, the subgraph of $\Gamma $ generated by $supp(\phi(s_1))$ has disconnected complement. By Lemma~\ref{complement}, that means the subgraph of $\Gamma $ generated by $supp(\phi(s_1))$ is a join of $supp(\phi(s_1))$ and $supp(\phi(s_2))$, which means elements of 
$supp(\phi(s_1))$ and $supp(\phi(s_2))$ pairwise commute.
\end{proof}

 \begin{remark}
We show that square maps are not limited to positive maps with the following example. Let $A(\Gamma) = \langle a, b, c | [a, c], [b, c] \rangle$:
 \begin{align*}
 \phi: & a \to aba^{-1}\\
         & b \to ba^{-1}\\
         & c \to c\\
\end{align*}
\end{remark}

\begin{lemma}\label{power}
Let $\phi$ be a square map. Then $\phi^{k}$ is also a square map for any positive integer $k$.
\end{lemma}

\begin{proof}
We prove the statement by induction on $k$. 
\[\phi^{k-1}(s_{1}) = s_{11}s_{12}...s_{1m} \qquad \text{ and } \qquad \phi^{k-1}(s_{2}) = s_{21}s_{22}...s_{2n} \]
Given that $s_{1i}$ and $s_{2j}$ commutes for all $i, j$  then under the square map $\phi$,  $supp(\phi(s_{1i}))$ and $supp(\phi(s_{1j}))$ pairwise commute for all $i, j$. That is to say, $supp(\phi^{k}(s_1))$ and $supp(\phi^{k}(s_{2}))$ pairwise commute.
\end{proof}

For the rest of the paper we assume $\phi$ is a pure, square automorphism.

\section{Automorphism Diagram}
The main tool of this paper is a directed graph $\calD_{\phi}$ that is defined on the generating set of the group $\A$. We deduce quantitative and structural information on the dilatation of $\phi$ from  the combinatorics of $\calD_{\phi}$.

\begin{definition}\label{autodiagram}
The vertex set of $\calD_{\phi}$ is $V$. Two vertices $s_i, s_j$ are connected by a directed edge from $s_i$ to $s_j$ if and only if $s_j \in \supp(s_i)$. 
\end{definition}

In this paper we are only concerned with \emph{pure} automorphisms; however, we do not add directed edges from $s_i$ to itself for any $i$. That is to say, only \emph{distinct} vertices $s_i, s_j$ can be connected by a directed edge in $\calD_{\phi}$. On the other hand, it is possible to have a directed edge from $s_i$ to $s_j$ and from $s_j$ to $s_i$.

For the rest of the section, assume that $\phi$ is a pure square map and $\phi \in \out(\A)$. Given a vertex $s \in \calD_{\phi}$, the \emph{down-set} of $s$ is an induced subgraph whose vertex set is the following union:
\[
\{ v | \text { there exists a directed path from } s \text{ to }v \}.
\]

The induced graph on this set of vertices we call $down-set$ and denote it by $d(s)$. Note that $s\in d(s)$, as the constant path is directed.  By construction,

\begin{equation}\label{dilatationofdownset}
\lambda_{s} = \lambda_{d(s)}.
\end{equation}

\begin{lemma}\label{Lem:path}
Suppose there is a directed path in $\calD_{\phi}$ of length $n$:
\[
(s_0, s_1, s_2,...s_n)
\]
and suppose that $s_0$ commutes with $s_1$. Then $\{s_0, s_1, s_2,...s_n\}$ form a complete subgraph in $\Gamma$.
\end{lemma}

\begin{proof}
First we show that $s_i$ and $s_{i+1}$ commute for all $i$.  Since $s_0$ and $s_1$ commute,  $\phi(s_0)$ and $\phi(s_1)$ commute. By assumption, $\phi$ is pure and a square map, 
%which means
%
%\[ |\phi(s_0) \phi(s_1)|  =   |\phi(s_0) | + | \phi(s_1)| \]
%By Corollary~\ref{support}, 

thus all elements of $\supp(\phi(s_1))$ commute with all elements of $\supp(\phi(s_2))$. In particular,
\[
s_1 \in \supp(\phi(s_0)) \text{ and } s_2 \in \supp(\phi(s_1)),
\]
thus, $s_1$ commutes with $s_2$. Applying this argument to $s_i$ and $s_{i+1}$ for all $i$, we  conclude that $s_i$ commutes with $s_{i+1}$ for all $i$. We finish the proof by proving the following statement:  Suppose $(s_0,\ldots,s_n)$ is a directed path in $\calD_\phi$ in which $s_i$ commutes with $s_{i+1}$ for all $0\leq i\leq n-1$, then $\{s_i\}$ form a complete subgraph. 

We prove this statement by induction on $n$.  The case $n=1$ is trivial. Now consider a path $(s_0,\ldots, s_n)$ such that $s_i$ and $s_{i+1}$ commute for all $i$. The subpaths $(s_0,\ldots, s_{n-1})$ and $(s_1,\ldots, s_n)$ have length $n-1$ hence by induction, it suffices to show that $s_0$ and $s_n$ commute.  Since $s_0$ and $s_{n-1}$ commute, by the argument above, every element in $\supp(\phi(s_0))$ commutes with every element in $\supp(\phi(s_{n-1}))$. Since $\phi$ is pure,  $s_0 \in \supp(\phi(s_0))$. Given that $s_n \in \supp(\phi(s_{n-1}))$,   $\phi$ is a square map means $s_0$ commutes with $s_n$. Therefore, all pairs of elements in $\{s_0, s_1, s_2,...s_n\}$ commute and $\{s_0, s_1, s_2,...s_n\}$ forms a complete subgraph in $\Gamma$.

\end{proof}

Now suppose instead of a directed path, there is a cycle $(s_0, s_1, s_2,...s_n,s_0)$ in $\calD_{\phi}$, 
\begin{corollary}\label{Cor:zgroupadjacent}
If an adjacent pair of elements of the cycle commute in $\A$, then 
\[
\{s_0, s_1, s_2,...s_n\}
\]
form a complete subgraph in $\Gamma$.
\end{corollary}
\begin{proof}
Suppose $s_{i}$ and $s_{i+1}$ commute, then consider the path 
\[
\{s_{i}, s_{i+1}, s_{i+2},...s_{n}, s_{1}, s_{2},...s_{i-1} \}
\]
is a directed path in $\calD_{\phi}$. By Lemma~\ref{Lem:path}, all elements on this path form a complete subgraph in $\calD_{\phi}$.
\end{proof}

\begin{corollary}\label{Cor:zgroup}
If a non-adjacent pair of elements of the cycle commute in $\A$, then 
\[
\{s_0, s_1, s_2,...s_k\}
\]
form a complete subgroup in $\Gamma$.
\end{corollary}
\begin{proof}
Without loss of generality, suppose $s_0$ commutes with $s_j$ for some $j \neq 2$.  We first claim that $\phi^j$ inherits the desirable 
properties of $\phi$:
\begin{claim}Let $\Phi=[\phi] \in \out(A(\Gamma))$ be a pure, square automorphism. Then $\phi^j, j\in \NN$ is also a pure, square automorphism.
\end{claim}
\noindent The fact that $\phi^{j}$ is a square automorphism follows from Remark~\ref{power}. The rest of the claim follows from Remark~ \ref{purepower}.

\end{proof}

From Corollary~\ref{Cor:zgroupadjacent} and Corollary~\ref{Cor:zgroup} we conclude that, 
\begin{proposition}\label{Prop:abelianorfree}
Every cycle $C \subseteq \calD_{\phi}$ forms an induced subgraph in $\Gamma$ that is either a complete graph or an empty graph.
\end{proposition}
\begin{proof}
Combine Corollary~\ref{Cor:zgroupadjacent}  and Corollary~\ref{Cor:zgroup}; if there is any pair of elements of $C$ that commutes then elements of $C$ form a complete subgraph in $\Gamma$. Otherwise no pair of elements of $C$ commute and the elements of $C$ form an empty graph in $\Gamma$.
\end{proof}

Moreover, if two cycles in $D_{\phi}$ share a vertex, then the union of all vertices support an empty or a complete subgraph in $\Gamma$:
\begin{corollary}\label{Lem:twocycles}
Let $s \in V$ be part of two (directed) cycles $C_1, C_2$, and there is a pair of commuting elements in $C_1 \cup C_2$, then the union of the vertices of $C_1$ and $C_2$ form a complete subgraph in $\Gamma$.
\end{corollary}
\begin{proof}
First suppose the two commuting vertices belong to one $C_i$ and without loss of generality that $i=1$. By Corollary~\ref{Cor:zgroup}, elements of $C_1$ form a complete subgraph. Every element of $C_1$ is connected by a directed path to every element of $C_2$. Indeed, let $s = s_{(1,1)} = s_{(2,1)}$. To connect the $i$-th element in $C_{1}$, $s_{(1, i)}$ and the $j$-th element in $C_{2}$, $s_{(2, j)}$, we have:
\[
s_{(1, i)}, s_{(1, i+1)}, s_{(1, i+2)}...s_{(1, 1)}=s_{(2,1)}, s_{(2,2)},...s_{(2, j)}.
\]
Therefore by Lemma~\ref{Lem:path}, every element of $C_1$ commutes with every element of $C_2$. It remains to consider elements of $C_2$.

It remains to consider the case when the commuting pair consists of one vertex from each cycle. That is, there exists $s_{(1, i)} \in C_1$ and $s_{(2, j)} \in C_2$ commute. There exists a cycle that contains $s_{(1, i)}$ and $s_{(2, j)}$. Indeed,
\[
s_{(1, i)}, s_{(1, i+1)}...s=s_{(2,1)}, s_{(2,2)},...s_{(2, j)}, s_{(2, j+1)}...s=s_{(1,1)}, s_{(1,2)}...s_{(1,i)}
\] 
forms a cycle. By Corollary~\ref{Cor:zgroup}, all elements in $C_1 \cup C_2$ form a complete subgraph.
\end{proof}
\begin{remark}
In short, it suffices to realize that 
\[s=s_{(1,1)}, s_{(1,2)}...s=s_{(2,1)}, s_{(2,2)},...s\]
is a cycle(with repeated vertex $s$), applying Corollary~\ref{Cor:zgroupadjacent}  and Corollary~\ref{Cor:zgroup} proves Corollary~\ref{Lem:twocycles}.
\end{remark}
Lastly, we use $\overline{\calD_{\phi}}$ to denote the underlying graph of $\calD_{\phi}$, that is, the graph whose vertex set is the vertex set of $\calD_{\phi}$, and $(v, v')$ is 
an (undirected) edge in $\overline{\calD_{\phi}}$ if there is a directed edge in $\calD_{\phi}$ either from $v$ to $v'$ or from $v'$ to $v$. We say that $\calD_{\phi}$ is \emph{connected} if 
$\overline{\calD_{\phi}}$ is connected. Otherwise we say that $\calD_{\phi}$ is \emph{disconnected}. 
We make the following observation about the dilatation of $\phi$ when $\calD_{\phi}$ is disconnected.

\begin{lemma}\label{component}
Suppose the connected components of $\calD_{\phi}$ is $(\calD_{\phi})_{1}, (\calD_{\phi})_{2},...(\calD_{\phi})_{k}$, then  there exists a $(\calD_{\phi})_{i}$ such that
\[
\lambda_{\phi} =  \lambda|_{(\calD_{\phi})_{i}}.
\]
\end{lemma}
\begin{proof}
For each generator $s$, by Equation~\ref{dilatationofdownset}, 
\[
\lambda_{s} = \lambda|_{d(s)}
\]
If $s \in (\calD_{\phi})_{j}$, then $d(s) \subset (\calD_{\phi})_{j}$, therefore 
\[ \lambda|_{(\calD_{\phi})_{j}} \geq \lambda|_{(\calD_{\phi})} = \lambda_{s} = \lambda_{\phi} \geq  \lambda|_{(\calD_{\phi})_{j}} .\]
Therefore 

\[
\lambda_{\phi} =  \lambda|_{(\calD_{\phi})_{j}}
\]
for the component that contains the dilatation-realizing generator. 
\end{proof}
\section{Growth of generators}
Given $\Gamma(A)$ and an automorphism $\phi$, we say a word $w$ grows \emph{polynomially} under iterations of $\phi$ if there exists $n_{0}$, a constant $C$ and a polynomial function $p(n)$ such that 
\[
|\phi^{n}(w)| \leq C p(n) \qquad \forall n >n_{0}.
\]
If $w=s_{0}$ is the generator that realizes the dilatation of $\phi$, and $s_{0}$ grows polynomially under the map $\phi$, or then we say that $\phi$ is an \emph{polynomially growing map}.
\begin{theorem}\label{Lem:cycle}
If there is no cycle in $\calD_{\phi}$,  then  $\phi$ is a polynomially growing map.
\end{theorem}
Before proving this theorem, we need some preliminary definitions for the case when $\calD_{\phi}$ does not have any cycle.
\begin{definition} A vertex $s\in \calD_{\phi}$ is \emph{terminal} if $d(s)=\emptyset$.
\end{definition}

There is a natural partition of the vertices of $s\in \calD_{\phi}$, defined as follows.  Let $\calT_0$ be the set of terminal vertices and inductively define $\calT_i$ to be the set of terminal vertices of \[\calD_{\phi}\setminus \cup_{j=0}^i\calT_j.\] 

\begin{definition}The collection $\calT=\{\calT_i\}$ is called the \textit{terminal partition} of $V$.  The \textit{height} of $\calT$ is the least $k$ such that $\cup_{j=0}^k\calT_j=V$.
\end{definition}

We now proceed with the proof of the theorem.
\begin{proof}
Suppose $\calD_{\phi}$ has no cycle. We will in fact show that $\phi$ has polynomial growth of degree bounded by the height $h$ of $\calT$. More precisely, we will show that there exists a polynomial $p(n)$ of degree at most $h$ such that for any $w\in A(\Gamma)$ there exists a constant $C$ such that \[|\phi^n(w)|\leq Cp(n).\]

Let $W_i=\cup_{j=0}^i\calT_j$ and let $G_i$ be the subgroup generated by $W_i$, and set $G_{-1}=\{1\}$. Clearly, $\phi$ preserves $G_i$. Choosing $s\in \calT_{i}$, we can write \[\phi(s)=t_0s^{\epsilon_1}\cdots s^{\epsilon_k}t_{k}.\]
where $\epsilon_i\in \ZZ$, and $t_j\in G_{i-1}$.
\begin{claim}\label{SimpleProduct}In the above decomposition, $k=1$, \emph{i.e.} $\phi(s)=t_0s^{\pm1}t_1$.
\end{claim}

\begin{proof}
The proof is by induction on the height $h$ of $\calT$.  For the base case let $\A$ be a right-angled Artin group with height zero, $h=0$. Then for each $s\in V$, $\phi(s)=s^m$ for all $s\in V$ and some $m\in \ZZ$. Since $\phi$ is an automorphism, $m=\pm1$.  Now assume the claim has been established for all RAAGs $A(\Gamma)$ and for any diagram of height at most $h$ without cycles, and suppose $\calD_\phi$ is a diagram of height $h+1$.  In particular, $G_{h+1}=A(\Gamma)$. The key observation is that $G_h$ is a special subgroup and that $\phi|_{G_h}=\phi_h$ is an automorphism of $G_h$ whose corresponding diagram has height $h$ and no cycle. By the induction hypothesis, for $0\leq i\leq h$ and for any $s\in \calT_i$, we have that \[\phi(s)=t_0 st_1\]
where $t_j\in G_{i-1}$. We show that there exists an automorphism $\psi$ such that $\tau=\psi^{-1}\circ\phi$ is an automorphism of height at most 1 and $\psi(s)=u_ssu_s^{-1}$ for all $s\in \calT_{h+1}$. Indeed, $\phi_h\in \Aut^0(G_h)$ and is therefore a product of partial conjugations and transvections.  On the other hand, $\phi_h$ is in the image of the restriction homomorphism $\rho_{h}:\Aut^0(G_{h+1}; G_h)\rightarrow \Aut^0(G_h)$.  We can lift $\phi_h$ to $\psi:\A\rightarrow \A$ at the expense of some element in $\ker(\rho_h)$. Observe that transvections lift directly, and partial conjugations lift at the expense of possibly conjugating elements of $\calT_{h+1}$ by elements of $W_h$.  Thus $\psi$ satisfies the following two properties: \begin{enumerate}
\item$\psi|_{G_h}=\phi_h$,
\item for every $s\in \calT_{h+1}$, we have\[\psi(s)=u_s su_s^{-1},\] where $u_s\in G_h$.
\end{enumerate}

Now $\tau=\psi^{-1}\circ \phi$ is a pure automorphism whose diagram $\calD_{\tau}$ also has no cycle, and which has height at most 1, since \[\tau|_{G_h}=\psi^{-1}|_{G_h}\circ\phi|_{G_h}=\phi_h^{-1}\circ\phi_h=\text{id}.\] By induction, we know that for every $s\in \calT_{h+1}$, we have $\tau(s)=t_0's^{\pm1}t_1'$, where $s$ is not contained in $\supp(t_0')$ nor $\supp(t_1')$.  From property (2), applying $\psi$ to $\tau(s)$, we have \[\phi(s)=\psi\circ\tau(s)=\psi(t_0's^{\pm1}t_1')=\psi(t_0')(u_s s^{\pm1}u_s^{-1})\psi(t_1').\]
As $s$ does not occur in the support of $\psi(t_0')$, $u_s$, or $\psi(t_1')$, we conclude that $k=1$ for all $s\in V$, as desired.  
\end{proof}

%\begin{claim} Elements of $G_i$ have growth bounded by a polynomial of degree $i$. 
%where $\epsilon_j=\pm1$ and each $t_j$ is a word in $G_i$
%\end{claim}
We now use Claim \ref{SimpleProduct} to show that there is a polynomial $p(n)$ such that for any $t\in G_i$ there exists $C>0$ with \[|\phi^n(t)|\leq Cp(n),\] where $p$ is a polynomial of degree at most $i$ The proof is by induction on $i$, where the base case is the observation above that $\phi(s)=s$ for $s\in \calT_0$. Suppose we have established the bound for all $t\in G_i$, and choose $s\in \calT_{i+1}$.  By Claim \ref{SimpleProduct} we can write \[\phi(s)=t_0st_1,\] where the $t_j$, $j=1,2$ are words in $G_i$. Let $C_j$, $j=1,2$ be constants such that \[|\phi^n(t_j)|=C_jp(n)\] and set $C=\max\{C_j\}$. Then we have \begin{align*}|\phi^n(s)|&\leq |\phi^{n-1}(t_0)|+|\phi^{n-1}(s)|+|\phi^{n-1}(t_1)|\\
&\leq 2Cp(n)+|\phi^{n-1}(s)|
\end{align*}
By iterating this procedure we have \[|\phi^n(s)|\leq 2C\sum_{r=1}^{n-1}p(r)+1.\]

We know that the partial sum of a polynomial of degree $d$ can be written as a polynomial of degree $d+1$. It follows that the right-hand side can be written as a polynomial of degree at most $i+1$.  Finally, an arbitrary word in $G_{i+1}$ can be written as a finite product of elements in some $\calT_j$, $0\leq j\leq i+1$. To finish the proof of the lemma, note that $G_i=A(\Gamma)$ for $i$ large.  

\end{proof}

\begin{corollary}\label{polygenerator}
Given a generator $s$ and $\phi$, if there is not a cycle in $d(s)$, then $s$ grows polynomially. 
\end{corollary}
\begin{proof}
This follows from Equation~\ref{dilatationofdownset}.
\end{proof}

\section{Invariant subgroups}

Now we are ready to prove the existence of a complete or empty subgraph inside a subgraph that is invariant under exponentially growing square maps.
\begin{theorem}\label{Thm:subgraph}
Let  $\Phi=[\phi] \in \out(\A)$ be a pure square map. If the dilatation $\lambda_{\phi}$ is positive, then there exists a $\phi$--invariant, induced subgraph $\Delta \subset \Gamma$ where 
$\lambda|_{\Delta} = \lambda_{\phi}$, and $\Delta$ is of one of the following cases:
\begin{enumerate}[(i)]
\item $\Delta$ is a complete graph.
\item $\Delta$ contains an empty graph $\Delta_{e}$ that is the union of all cycles in $\calD_{\Delta}$.
\end{enumerate}
\end{theorem}

\begin{proof}
By Lemma~\ref{Lem:dilatationofgenerator}, there exists a generator $s$ that realizes the dilatation of the group, let's denote it $s_{0}$:
\[
\lambda_{s_{0}} = \lambda_{\phi}.
\]
Consider  the down-set of $s_{0}$ in $\calD_{\phi}$, recall $\lambda_{d(s_{0})}$ is well defined and 
$\lambda_{d(s_{0})}  = \lambda_{s_{0}} = \lambda_{\phi}$. By assumption there exists a cycle in $\calD_{d(s_{0})}$. Now consider a ``trim'' of $d(s_{0})$: a directed edge is \emph{outgoing} with respect to a vertex $v$ if the edge is from $v$ to $v'$ for some other $v'$. If a vertex in a directed graph has only outgoing edges, then we say the vertex is a \emph{source}. If a vertex has only incoming edges, then we say the vertex is a \emph{sink}. Consider 
the vertices of $d(s_{0})$ and consider the set of all the non-source vertices:
\[
[d(s_{0})]_{1} := \{ v | v \in d(s_{0}) \text{ and } v \text{ is not a source }\}.
\]
By an abuse of notation we also use $[d(s_{0})]_{1}$ to denote the induced directed graph on this set. $[d(s_{0})]_{1}$ is $\phi$--invariant. Now we trim $[d(s_{0})]_{1}$ likewise: subtract from $[d(s_{0})]_{1}$ any and all of its source vertices and all directed edges incident to the source vertices. The trimming process ends after finite steps since there exists a cycle in $d(s_{0})$. The resulting directed graph we denote by
$[d(s_{0})]_{t}$, which is also $\phi$--invariant and contains a cycle. Since $[d(s_{0})]_{t}$ is $\phi$--invariant, one can consider the dilatation of $\phi$ restricted to $[d(s_{0})]_{t}$. It is noted that the trimming process does not change the dilatation, since a word of arbitrary length visit a source vertex a finite number of times. Thus,
\[
\lambda_{\phi} = \lambda_{s_{0}} =\lambda_{d(s_{0})} =\lambda_{[d(s_{0})]_{1}} = \lambda_{[d(s_{0})]_{2}} =... =\lambda_{[d(s_{0})]_{t}}.
\]

Suppose $[d(s_{0})]_{t}$ has exactly one component in $\calD_{\phi}$. Let $\Delta: = [d(s_{0})]_{t}$. 
Since $\Delta$ contains at least one cycle, let $\Delta_{e}$ denote the subgraph of $\Delta$ that is the union of all cycles. $\Delta_{e}$ is unique by maximality. If a pair of elements in $\Delta_{e}$ commutes, then by Corollary~\ref{Cor:zgroupadjacent} and Corollary~\ref{Cor:zgroup}, $\Delta_{e}$ is a complete graph. Furthermore, consider a vertex $v \in \Delta$ and $v \notin \Delta_{e}$. Since $\Delta$ is source-less and connected, there exists a directed path starting from an edge in $\Delta_{e}$ and ending at $v$. Therefore by Lemma~\ref{Lem:path}, $v$ forms a complete graph with $\Delta_{e}$. Moreover, we show that two such vertices necessarily commute in the following lemma:
\begin{lemma}\label{Lem:comPersist}
Let $\phi$ be a pure, square map. Let $A, B$ be a a pair of commuting vertices and let $v, v'$ be two vertices that are each connected by a directed path \emph{from} $A$ and $B$, respectively. Then $v$ commutes with $v'$ in $\A$.  
\end{lemma}

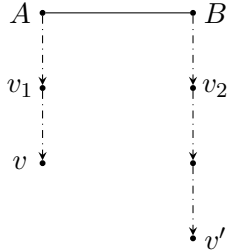
\begin{figure}[h!]
\begin{tikzpicture}
 \tikzstyle{vertex} =[circle,draw,fill=black,thick, inner sep=0pt,minimum size=.5 mm]
 
[thick, 
    scale=1,
    vertex/.style={circle,draw,fill=black,thick,
                   inner sep=0pt,minimum size= .5 mm},
                  
      trans/.style={thick,->, shorten >=6pt,shorten <=6pt,>=stealth},
   ]
   
\node[vertex] (a) at (0, 0) {};
\node[vertex] (b) at (2,0) {};
\node[vertex] (a1) at (0, -1) {};
\node[vertex] (a2) at (0, -2) {};
\node[vertex] (b1) at (2,-1) {};
\node[vertex] (b2) at (2,-2) {};
\node[vertex] (b3) at (2,-3) {};
   
\draw [-](a)--(b);
\draw[dash dot, ->](a)--(a1);
\draw[dash dot, ->](a1)--(a2);
\draw[dash dot, ->](b)--(b1);
\draw[dash dot, ->](b1)--(b2);
\draw[dash dot, ->](b2)--(b3);

\node at (-0.3, 0){$A$};
\node at (2.3, 0){$B$};
\node at (-0.3, -1){$v_{1}$};
\node at (2.3, -1){$v_{2}$};
\node at (-0.3, -2){$v$};
\node at (2.3, -3){$v'$};

\end{tikzpicture}
\caption*{Vertices that are in the down-set of two commuting vertices commute.}
\end{figure}
\begin{proof}
Let $v_{1}$ be the first vertex on the directed path from $A$ to $v$ and $v_{2}$ be the first vertex on the directed path from $B$ to $v'$. By an abuse of notation, we use $v, v_{1}, v_{2}, etc.$ to denote also the generator that they represent in $\A$. Since $v_{1} \in supp(\phi(A))$ and $B \in supp(\phi(B))$, by Definition~\ref{Def:Square}, $v_{1}$ and $B$ commute. Likewise $v_{1}$ and $v_{2}$ commute. Using  Definition~\ref{Def:Square}, we can show that $v_{1}$ commute with all successive elements of the directed path originating from $B$, therefore $v_{1}$ commute with $v'$. Now consider the next vertex on the direct path from $A$, call it $v_{11}$. Using  Definition~\ref{Def:Square} again one can show that $v_{11}$
commute with all successive elements of the directed path originating from $B$ and hence $v'$. By moving down the directed path from $A$, one show that $v$ and $v'$ commute.
\end{proof}

Therefore all vertices that are in $\Delta \setminus \Delta_{e}$ commute to each other, therefore, $\Delta$ is a complete subgraph in $\Gamma$. On the other hand, if there is no edge in $\Delta_{e}$ then $\Delta$ contains an empty graph $\Delta_{e}$ that is the union of all cycles in $\calD_{\Delta}$.

It remains to consider the case when $\Delta$ is not connected. In this case, each connected component of $[d(s_{0})]_{t}$ is either a complete graph or contains an empty graph that is a union of all cycles in that component in $\calD_{\phi}$. By Lemma~\ref{component}, there exists a component of $\Delta$ with desired property that realizes the dilatation.

\end{proof}

\bibliographystyle{alpha}
\bibliography{DilatationBib}

\begin{thebibliography}{CCV07}

\bibitem[Ago13]{Ag13}
Ian Agol.
\newblock The virtual {H}aken conjecture.
\newblock {\em Doc. Math.}, 18:1045--1087, 2013.
\newblock With an appendix by Agol, Daniel Groves, and Jason Manning.

\bibitem[BCV09]{BCV09}
Kai-Uwe Bux, Ruth Charney, and Karen Vogtmann.
\newblock Automorphisms of two-dimensional {RAAGS} and partially symmetric
  automorphisms of free groups.
\newblock {\em Groups Geom. Dyn.}, 3(4):541--554, 2009.

\bibitem[BFH00]{BFH00}
Mladen Bestvina, Mark Feighn, and Michael Handel.
\newblock The {T}its alternative for {${\rm Out}(F_n)$}. {I}. {D}ynamics of
  exponentially-growing automorphisms.
\newblock {\em Ann. of Math. (2)}, 151(2):517--623, 2000.

\bibitem[BFH05]{BFH05}
Mladen Bestvina, Mark Feighn, and Michael Handel.
\newblock The {T}its alternative for {${\rm Out}(F_n)$}. {II}. {A} {K}olchin
  type theorem.
\newblock {\em Ann. of Math. (2)}, 161(1):1--59, 2005.

\bibitem[BH92]{BH92}
Mladen Bestvina and Michael Handel.
\newblock Train tracks and automorphisms of free groups.
\newblock {\em Ann. of Math. (2)}, 135(1):1--51, 1992.

\bibitem[CCV07]{CCV07}
Ruth Charney, John Crisp, and Karen Vogtmann.
\newblock Automorphisms of 2-dimensional right-angled {A}rtin groups.
\newblock {\em Geom. Topol.}, 11:2227--2264, 2007.

\bibitem[CSV17]{CSV17}
Ruth Charney, Nathaniel Stambaugh, and Karen Vogtmann.
\newblock Outer space for untwisted automorphisms of right-angled {A}rtin
  groups.
\newblock {\em Geom. Topol.}, 21(2):1131--1178, 2017.

\bibitem[CV86]{CV86}
Marc Culler and Karen Vogtmann.
\newblock Moduli of graphs and automorphisms of free groups.
\newblock {\em Invent. Math.}, 84(1):91--119, 1986.

\bibitem[CV09]{ChVo09}
Ruth Charney and Karen Vogtmann.
\newblock Finiteness properties of automorphism groups of right-angled {A}rtin
  groups.
\newblock {\em Bull. Lond. Math. Soc.}, 41(1):94--102, 2009.

\bibitem[Day09]{Day09}
Matthew~B. Day.
\newblock Peak reduction and finite presentations for automorphism groups of
  right-angled {A}rtin groups.
\newblock {\em Geom. Topol.}, 13(2):817--855, 2009.

\bibitem[DW17]{DW18}
Matthew~B. Day and Richard~D. Wade.
\newblock Relative automorphism groups of right-angled artin groups.
\newblock Preprint: https://arxiv.org/abs/1712.01583, 2017.

\bibitem[KK17]{KK17}
Sang-hyun Kim and Thomas Koberda.
\newblock R{AAG}s in diffeos.
\newblock In {\em Hyperbolic geometry and geometric group theory}, volume~73 of
  {\em Adv. Stud. Pure Math.}, pages 215--224. Math. Soc. Japan, Tokyo, 2017.

\bibitem[Lau95]{Lau95}
Michael~R. Laurence.
\newblock A generating set for the automorphism group of a graph group.
\newblock {\em J. London Math. Soc. (2)}, 52(2):318--334, 1995.

\bibitem[Lev09]{LE09}
Gilbert Levitt.
\newblock Counting growth types of automorphisms of free groups.
\newblock {\em Geom. Funct. Anal.}, 19(4):1119--1146, 2009.

\bibitem[LS77]{LS77}
Roger~C. Lyndon and Paul~E. Schupp.
\newblock {\em Combinatorial group theory}.
\newblock Springer-Verlag, Berlin-New York, 1977.
\newblock Ergebnisse der Mathematik und ihrer Grenzgebiete, Band 89.

\bibitem[Ser89]{Ser89}
Herman Servatius.
\newblock Automorphisms of graph groups.
\newblock {\em J. Algebra}, 126(1):34--60, 1989.

\end{thebibliography}

\end{document}